\documentclass[10pt,letterpaper]{article}
\makeatletter
\usepackage{amscd}
\usepackage{amsmath}
\usepackage{latexsym}
\usepackage{amsfonts}
\usepackage{amssymb}
\usepackage{amsthm}
\usepackage{graphicx}
\usepackage[dvipdfm,
            pdfstartview=FitH,
            CJKbookmarks=true,
            bookmarksnumbered=true,
            bookmarksopen=true,
            colorlinks=true, 
            pdfborder=001,   
            citecolor=magenta,
            linkcolor=blue,
            ]{hyperref}

\newtheorem{thm}{\textbf Theorem}[section]
\newtheorem{lem}{\textbf Lemma}[section]
\newtheorem{rem}{\textbf Remark}[section]

\newcommand{\md}{\mbox{d}}
\newcommand{\be}{\begin{eqnarray}}
\newcommand{\ee}{\end{eqnarray}}
\newcommand{\mr}{\mathbb{R}}

\newcommand{\bes}{\begin{eqnarray*}}
\newcommand{\ees}{\end{eqnarray*}}

\newcommand{\la}{\lambda}

\begin{document}
\begin{titlepage}
\title{\bf The global well-posedness of strong solutions to 2D MHD equations in Lei-Lin space}
\author{ Baoquan Yuan\thanks{Corresponding Author: B. Yuan}\ and Yamin Xiao
       \\ School of Mathematics and Information Science,
       \\ Henan Polytechnic University,  Henan,  454000,  China.\\
        (bqyuan@hpu.edu.cn, ymxiao106@163.com)
          }

\date{}
\end{titlepage}
\maketitle
\begin{abstract}
In this paper, we study the Cauchy problem of the 2D incompressible magnetohydrodynamic equations in Lei-Lin space. The global well-posedness of a strong solution in the Lei-Lin space $\chi^{-1}(\mr^2)$ with any initial data in $\chi^{-1}(\mr^2)\cap L^2(\mr^2)$ is established. Furthermore, the uniqueness of the strong solution in $\chi^{-1}(\mr^2)$ and the Leray-Hopf weak solution in $L^2(\mr^2)$ is proved.
\end{abstract}

\vspace{.2in} {\bf Key words:} 2D MHD equations, strong solutions, Lei-Lin space, weak-strong uniqueness.

\vspace{.2in} {\bf MSC(2010):} 76W05, 35D35, 74H25.



\section{Introduction}
\setcounter{equation}{0}
\vskip .1in
In this paper, we consider the following incompressible magnetohydrodynamic equations
\be\label{1.1}
\begin{cases}
{ \begin{array}{ll}
\partial_{t}u-\mu\Delta u+(u\cdot\nabla) u+\nabla p=(b\cdot\nabla) b, \\
\partial_{t}b-\nu\Delta u+(u\cdot\nabla) b=(b\cdot\nabla) u, \\
\nabla\cdot u=\nabla\cdot b=0, \\
u(0,x)=u_{0}(x), \ b(0,x)=b_{0}(x).
 \end{array} }
\end{cases}
\ee
for $t\geq0$, $x\in\mr^2$. We denote $u=u(t,x)$, $b=b(t,x)$ and $p=p(t,x)$ the velocity field,
magnetic field and scalar pressure of the fluid, respectively. The positive constants $\mu$ and $\nu$ are the viscosity and the resistivity coefficients, and $u_0(x)$ and $b_0(x)$ are the initial velocity and initial magnetic field satisfying $\mbox{div}u_0=\mbox{div}b_0=0$, respectively.

When $b=0$, it reduces to the classical Navier-Stokes equation, which has been investigated extensively with many interesting results. Leray \cite{L} and Hopf \cite{H} established the global existence of weak solutions. Kato \cite{K}, Fujita and Kato \cite{F-K} obtained the local well-posedness for any initial data and the global well-posedness for small initial data in $L^n(\mr^n)$ and $\dot{H}^{s}(\mr^n)$ with $s\geq\frac{n}{2}-1$, respectively. Recently, Lei and Lin \cite{L-L} constructed the global mild solution with small initial data in the critical space $\chi^{-1}(\mr^3)$, and Zhang and Yin \cite{Z-Y} studied the local well-posedness with large initial data and the  global well-posedness with small initial data by the semi-group method. \\
{\bf Theorem A} (\cite{L-L} and \cite{Z-Y}) Let $u_0$ be in $\chi^{-1}(\mr^3)$. There exists a positive time $T$ such that the Navier-Stokes equation has a unique solution $u$ in $L^2([0,T];\chi^0(\mr^3))$ which also belongs to
$$\mathcal{C}([0,T];\chi^{-1}(\mr^3))\cap L^1([0,T];\chi^1(\mr^3))\cap L^\infty([0,T];\chi^{-1}(\mr^3)).$$
Let $T_{u_0}$ denote the maximal time of existence of such a solution. Then:\\
-There exists a constant $C$ such that, if $\|u_0\|_{\chi^{-1}}\leq C\mu$, then $$T_{u_0}=\infty.$$
-If $T_{u_0}$ is finite, then $$\int_0^{T_{u_0}}\|u\|_{\chi^{0}}^2\md t=\infty.$$

There are several important global well-posedness and decay results for the 3D Navier-Stokes equation and MHD equations (\ref{1.1}) in Lei-Lin space $\chi^{-1}(\mr^3)$ (see, e.g., \cite{B1,B2,B-B,W,W-W,X-Y-Z,Y}). In particular, Ye and Zhao \cite{Y-Z} proved the global well-posedness for the $n$ dimensional generalized MHD equations with small initial data. As we know that, in 2D case, the global Leray-Hopf weak solution to the MHD equations (\ref{1.1}) is a unique one in $L^2(\mr^2)$ space. We wonder if there is a global strong solution to the 2D MHD equations (\ref{1.1}) in the Lei-Lin space $\chi^{-1}(\mr^2)$ without the smallness condition, and if it is a unique solution with the Leray-Hopf weak solution with a same initial data in the space $L^2(\mr^2)\cap\chi^{-1}(\mr^2)$.

In this paper, we will prove the global well-posedness of a strong solution to (\ref{1.1}) in $\chi^{-1}(\mr^2)\cap L^2(\mr^2)$ for large initial data. Our main results are presented as follows.
\begin{thm}\label{thm1.1}
Let $(u_0,b_0)$ be in $\chi^{-1}(\mr^2)\cap L^2(\mr^2)$. A unique solution $(u,b)$ then exists in the space $L^2(\mr^+;\chi^0(\mr^2))$ which also belongs to
$$\mathcal{C}(\mr^+;\chi^{-1}(\mr^2))\cap L^1(\mr^+;\chi^1(\mr^2))\cap L^\infty(\mr^+;L^2(\mr^2))\cap L^2(\mr^+;\dot{H}^1(\mr^2)),$$
and satisfies the energy equality
\begin{align}\label{1.2}
\|(u,b)\|_{L^2}^2+2\mu\int_0^t\|\nabla u\|_{L^2}^2\md \tau+2\nu\int_0^t\|\nabla b\|_{L^2}^2\md \tau=\|(u_0,b_0)\|_{L^2}^2.
\end{align}
$(u,b)$ also satisfies the global a priori estimate
\begin{align}\label{1.3}
\|(u,b)\|_{\widetilde{L}^\infty(\chi^{-1})}+\frac{1}{2}\min\{\mu,\nu\}\int_0^t\|(u,b)\|_{\chi^1}\md\tau
\leq\|(u_0,b_0)\|_{\chi^{-1}}+\frac{C}{2\min\{\mu\nu\}}\|(u_0,b_0)\|_{L^2}^4.
\end{align}
\end{thm}
The weak-strong uniqueness theorem is in order.
\begin{thm}\label{thm1.2}
Let $(u_0,b_0)$ and $(v_0,h_0)$ be the divergence-free vector field in $L^2(\mr^2)$ and $(u_0,b_0)$ also be in $\chi^{-1}(\mr^2)$. Let $(v,h)\in L^\infty(\mr^+;L^2(\mr^2))\cap L^2(\mr^+;\dot{H}^1(\mr^2))$ be a Leray-Hopf weak solution associated with $(v_0,h_0)$, and let $(u,b)$ be a strong solution constructed in Theorem \ref{thm1.1} associated with $(u_0,b_0)$ and $$(u,b)\in L^2(\mr^+;\chi^0(\mr^2))\cap L^\infty(\mr^+;L^2(\mr^2))\cap L^2(\mr^+;\dot{H}^1(\mr^2)).$$
Write $(w,g)\triangleq (u-v,b-h)$, then we have for $0\leq t<\infty$
\begin{align}\label{1.4}
\nonumber&\|(w,g)\|_{L^2}^2+\frac12\min\{\mu,\nu\}\left(\int_0^t\|\nabla w\|_{L^2}^2\md \tau+\int_0^t\|\nabla g\|_{L^2}^2\md \tau\right)\\
\leq&\|(u_0,b_0)-(v_0,h_0)\|_{L^2}\times\mbox{exp}\left(C\int_0^t\|(u,b)\|_{\chi^0}^2\md \tau\right).
\end{align}
\end{thm}
\begin{rem}
If two solutions $(u,b)$ and $(v,h)$ constructed in Theorem \ref{thm1.1} have a same initial data $(u_0,b_0)=(v_0,h_0)\in L^2(\mr^2)\cap \chi^{-1}(\mr^2)$, then the estimate inequality (\ref{1.4}) implies the uniqueness of solutions $(u,b)\equiv (v,h)$.
\end{rem}
\begin{rem}
The definitions and notations used in this paper are presented as follows:\\
(1) For $s\in \mr$, the functional space $\chi^s(\mr^n)$ is defined by
$$\chi^{s}:=\left\{f\in \mathcal{D}'(\mr^n)\ \bigg| \ \int_{\mr^n}|\xi|^s|\hat f(\xi)|\mathrm{d} \xi<\infty\right\}.$$
(2) Let $s\in\mr$ and $p\in [1,\infty]$. $f(t,x)\in L^p([0,T];\chi^s(\mr^n))$ if and only if
$$\|f\|_{L^p(\chi^s)}\triangleq\Big( \int^T_0\|f(t,\cdot)\|^p_{\chi^s}\md t\Big)^{\frac1{p}}<\infty;$$
$f(t,x)\in \tilde L^p([0,T];\chi^s(\mr^n))$ if and only if
$$\|f\|_{\tilde L^p(\chi^s)}\triangleq\Big\|\Big(\int^T_0\left(|\xi|^s|\hat{f}(t,\xi)|\right)^p\md t\Big)^{\frac1{p}}\Big\|_{L^1}<\infty.$$
(3) we will use $\|\cdot\|_{X}$ to denote $\|\cdot\|_{X(\mr^n)}$, $\|\cdot\|_{\tilde{L}^q(\chi^s)}$ to denote $\|\cdot\|_{\tilde{L}^q([0,T];\chi^s(\mr^n))}$ and $\|(u,b)\|^p_{X}$ to denote $\|u\|^p_{X}+\|b\|^p_{X}$ for conciseness. And throughout the paper, $C$ stands for a generic positive constant, which may be different from line to line.
\end{rem}
\begin{rem}
The space $L^2(\mr^2)\cap\chi^{-1}(\mr^2)$ is not empty. Indeed, let $f=\mathcal{F}^{-1}\left(\frac{1}{|\xi|^2}\mathbf{1}_{\{|\xi|>2\}}\right)$, then we have
$$\|f\|_{\chi^{-1}}=2\pi\int_2^{+\infty}\frac{1}{r^2}\md r<\infty$$ and $$\|f\|_{L^2}=2\pi\int_2^{+\infty}\frac{1}{r^3}\md r<\infty.$$
Therefore, $f\in L^2(\mr^2)\cap\chi^{-1}(\mr^2)$.
\end{rem}


\section{Preliminaries}
\setcounter{equation}{0}
\vskip .1in
In this Preliminary section, we present some elementary lemmas which will be used in our proofs.
\begin{lem}\label{lem2.1}
Let $s_1<s_0<s_2$. If $f\in \chi^{s_1}(\mr^n)\cap\chi^{s_2}(\mr^n)$, then $f\in\chi^{s_0}(\mr^n)$, and
$$\|f\|_{\chi^{s_0}}\leq\|f\|_{\chi^{s_1}}^{\frac{s_2-s_0}{s_2-s_1}}\|f\|_{\chi^{s_2}}^{\frac{s_0-s_1}{s_2-s_1}}.$$
\begin{proof}
\begin{align}\label{2.1}
\nonumber\|f\|_{\chi^{s_0}}=&\int_{\mr^n}|\xi|^{s_0}|\hat{f}(\xi)|\md \xi\\
\nonumber=&\int_{|\xi|\leq\la}|\xi|^{s_0-s_1}|\xi|^{s_1}|\hat{f}(\xi)|\md \xi+\int_{|\xi|>\la}|\xi|^{s_0-s_2}|\xi|^{s_2}|\hat{f}(\xi)|\md \xi\\
\leq&\la^{s_0-s_1}\|f\|_{\chi^{s_1}}+\la^{s_0-s_2}\|f\|_{\chi^{s_2}}.
\end{align}
Let $\la=\left(\frac{\|f\|_{\chi^{s_2}}}{\|f\|_{\chi^{s_1}}}\right)^{\frac{1}{s_2-s_1}}$, then the Lemma \ref{lem2.1} follows from (\ref{2.1}).
\end{proof}
\end{lem}
\begin{lem}\label{lem2.2}
Let $f\in L^2(\mr^2)\cap\dot{H}^1(\mr^2)$, then $f\in\chi^{-\frac12}(\mr^2)$, and $\|f\|_{\chi^{-\frac12}}\leq C\|f\|_{L^2}^{\frac12}\|f\|_{\dot{H}^1}^{\frac12}$.
\end{lem}
\begin{proof}
We write
\begin{align}\label{2.2}
\|f\|_{\chi^{-\frac12}}=\int_{\mr^2}|\xi|^{-\frac12}|\hat{f}(\xi)|\md \xi:=I_1(\lambda)+I_2(\lambda),
\end{align}
where
$$I_1(\lambda)=\int_{|\xi|\leq\lambda}|\xi|^{-\frac12}|\hat{f}(\xi)|\md \xi\leq C\lambda^{\frac12}\|f\|_{L^2}$$
and
$$I_2(\lambda)=\int_{|\xi|>\lambda}|\xi|^{-\frac12-1}|\xi||\hat{f}(\xi)|\md \xi\leq C\lambda^{-\frac12}\|f\|_{\dot{H}^{1}}$$
by the H\"{o}lder's inequality. Choosing $\lambda=\frac{\|f\|_{\dot{H}^{1}}}{\|f\|_{L^2}}$ completes the proof of Lemma \ref{lem2.2}.
\end{proof}
\begin{lem}\label{lem2.3}
Let $f,g\in L^2([0,T];\chi^0(\mr^n))$, then $fg\in L^1([0,T];\chi^0(\mr^n))$ and
$$\|fg\|_{L^1(\chi^{0})}\leq \|f\|_{L^2(\chi^{0})}\|g\|_{L^2(\chi^{0})}.$$
In particular, $\|f^2\|_{L^1(\chi^{0})}\leq \|f\|_{L^2(\chi^{0})}^2$.
\end{lem}
\begin{proof}
\begin{align}\label{2.3}
\nonumber\|fg\|_{L^1(\chi^{0})}=&\int_0^T\int_{\mr^n}|\hat{f}(\tau,\xi)|\ast_\xi|\hat{g}(\tau,\xi)|\md \xi\md \tau\\
\nonumber\leq &\int_0^T\|f\|_{\chi^0}\|g\|_{\chi^0}\md \tau\\
\leq &\|f\|_{L^2(\chi^0)}\|g\|_{L^2(\chi^0)}.
\end{align}
\end{proof}
\begin{lem}\label{lem2.4}
There exists a constant $C$ such that
$$\|B(u,b)\|_{L^2(\chi^0)}\leq \frac{C}{\min\{\mu,\nu\}^\frac12}\|(u,b)\|_{L^2(\chi^0)}^2,$$
where $B(u,b)\triangleq\int_0^t[\mathrm{e}^{\mu(t-\tau)\triangle}\mathbb{P}\nabla\cdot(u\otimes u+b\otimes b)+\mathrm{e}^{\nu(t-\tau)\triangle}\nabla\cdot(u\otimes b+b\otimes u)]\md \tau$.
\end{lem}
\begin{proof}
\begin{align}\label{2.4}
\nonumber&\|B(u,b)\|_{L^2(\chi^0)}\\
\nonumber\leq&\left\|\int_{\mr^2}\int_0^{T}\mathrm{e}^{-\min\{\mu,\nu\}(t-\tau)|\xi|^{2}}|\xi|\left(|\hat{u}|\ast|\hat{u}
|+|\hat{b}|\ast|\hat{b}|+|\hat{u}|\ast|\hat{b}
|+|\hat{b}|\ast|\hat{u}|\right)\md \tau\md\xi\right\|_{L^2}\\
\nonumber\leq&\int_{\mr^2}\left(\int_0^{T}\mathrm{e}^{-2\min\{\mu,\nu\}\tau|\xi|^{2}}|\xi|^{2}\md \tau\right)^{\frac12}\int_0^T\left(|\hat{u}|\ast|\hat{u}|+|\hat{b}|\ast|\hat{b}|+|\hat{u}|\ast|\hat{b}
|+|\hat{b}|\ast|\hat{u}|\right)\md \tau\md\xi\\
\nonumber\leq&\frac{C}{\min\{\mu,\nu\}^\frac12}\int_0^T\left(\|u\|_{\chi^0}^2+\|b\|^2_{\chi^0}+2\|u\|_{\chi^0}\|b\|_{\chi^0}\right)\md\tau\\
\leq&\frac{C}{\min\{\mu,\nu\}^\frac12}\|(u,b)\|_{L^2(\chi^0)}^2.
\end{align}
\end{proof}
We also need the following Banach contraction mapping principle (see \cite[Lemma 5.5]{B-C-D}).
\begin{lem}\label{lem2.5}
Let $E$ be a Banach space, $\mathcal{B}$ a continuous bilinear map from $E\times E$ to $E$, and $\alpha$ a
positive real number such that
\bes
\alpha<\frac{1}{4\|\mathcal{B}\|} \ with \ \|\mathcal{B}\|=\sup_{\|f\|,\|g\|\leq1}\|\mathcal{B}(f,g)\|.
\ees
For any $a$ in the ball $B(0,\alpha)$ (i.e., with center $0$ and radius $\alpha$) in $E$, a unique $x$ then
exists in $B(0,2\alpha)$ such that
\bes
x=a+\mathcal{B}(x,x).
\ees
\end{lem}
Finally, we show some estimates of the heat equations in the space $\mathcal{C}([0,T];\chi^{s}(\mr^n))\cap L^1((0,T];\chi^{s+2}(\mr^n))$.
\begin{lem}\label{lem2.6}
For $s\in\mr$, let $v$ be a solution of the Cauchy problem
\be\label{2.5}
\begin{cases}
{ \begin{array}{ll}
\partial_{t}v-\kappa\Delta v=f, \\
v(x,0)=v_0(x),
 \end{array} }
\end{cases}
\ee
with $f\in L^1([0,T];\chi^{s}(\mr^n))$ and $v_0\in\chi^s(\mr^n)$. Then $v$ belongs to $$\widetilde{L}^\infty([0,T];\chi^{s}(\mr^n))\cap L^1((0,T];\chi^{s+2}(\mr^n))\cap\mathcal{C}([0,T];\chi^{s}(\mr^n)),$$ and satisfies the following estimate
\begin{align}\label{2.6}
\|v\|_{\widetilde{L}^\infty(\chi^{s})}+\kappa\|v\|_{L^1(\chi^{s+2})}\leq C\left(\|v_0\|_{\chi^s}+\|f\|_{L^1(\chi^s)}\right).
\end{align}
\end{lem}
\begin{proof}
By the Duhamel's formula in Fourier space, the equations (\ref{2.5}) can be written as
\begin{align}\label{2.7}
|\hat{v}(t,\xi)|\leq\mathrm{e}^{-\kappa t|\xi|^{2}}|\hat{v}_0(\xi)|+\int_0^t\mathrm{e}^{-\kappa(t-\tau)|\xi|^{2}}|\hat{f}(\tau,\xi)|\md \tau.
\end{align}
Multiplying (\ref{2.7}) by $|\xi|^s$ and taking the $L^\infty$ norm in time, one has
\begin{align}\label{2.8}
\sup_{0\leq t\leq T}|\xi|^{s}|\hat{v}(t,\xi)|\leq|\xi|^{s}|\hat{v}_0(\xi)|+\int_0^t|\xi|^s|\hat{f}(\tau,\xi)|\md \tau.
\end{align}
Taking the $L^1$ norm in $\xi$ to (\ref{2.8}), we get
\begin{align}\label{2.9}
\nonumber\|v\|_{\widetilde{L}^\infty(\chi^{s})}\leq&\int_{\mr^n}|\xi|^{s}|\hat{v}_0(\xi)|\md\xi
+\int_{\mr^n}\int_0^t|\xi|^s|\hat{f}(\tau,\xi)|\md \tau\md \xi\\
\leq&\|v_0\|_{\chi^s}+\|f\|_{L^1(\chi^s)}.
\end{align}
Because, for almost all fixed $\xi\in\mr^n$, the map $t\mapsto\hat{v}(t,\xi)$ is continuous over $[0,T]$, thus by the Lebesgue dominated convergence theorem, it implies that $v\in\mathcal{C}([0,T];\chi^s(\mr^n))$. Next we estimate $\|v\|_{L^1(\chi^{s+2})}$. Multiplying (\ref{2.7}) by $|\xi|^{s+2}$ and taking the $L^1$-norm in time, applying the Young's inequality in time and noting $\int_0^T |\xi|^{2}\mathrm{e}^{-\kappa t|\xi|^{2}}\md t$
$=\frac{1}{\kappa}$, we deduce that
\begin{align}\label{2.10}
\nonumber\int_0^T|\xi|^{s+2}|\hat{v}(t,\xi)|\md t\leq&\int_0^T\mathrm{e}^{-\kappa t|\xi|^{2}}|\xi|^{s+2}|\hat{v}_0(\xi)|\md t+\int_0^T\int_0^t\mathrm{e}^{-\kappa(t-\tau)|\xi|^{2}}|\xi|^{s+2}|\hat{f}(\tau,\xi)|\md \tau\md t\\
\leq&\frac1\kappa\left(|\xi|^{s}|\hat{v}_0(\xi)|+\int_0^T|\xi|^s|\hat{f}(t,\xi)|\md t\right),
\end{align}
then taking $L^1$ norm in $\xi$ to (\ref{2.10}), we obtain
\begin{align}\label{2.11}
\nonumber\|v\|_{L^1(\chi^{s+2})}\leq&\frac1\kappa\left(\int_{\mr^n}|\xi|^{s}|\hat{v}_0(\xi)|\md\xi
+\int_{\mr^n}\int_0^T|\xi|^s|\hat{f}(t,\xi)|\md t\md \xi\right)\\
\leq&\frac1\kappa\left(\|v_0\|_{\chi^s}+\|f\|_{L^1(\chi^s)}\right).
\end{align}
Combining (\ref{2.9}) and (\ref{2.11}), we get (\ref{2.6}).
\end{proof}

\section{Proof of Theorems \ref{thm1.1} and \ref{thm1.2}}
In this section, we focus on the global well-posedness of a mild solution to (\ref{1.1}) in $L^2(\mr^+;\chi^{0}(\mr^2))$. First, we give the following lemma on the well-posedness and blow-up criterion, which plays a key role in proving our theorem.
\setcounter{equation}{0}
\begin{lem}\label{lem3.1}
Assume $(u_0,b_0)$ be in $\chi^{-1}(\mr^2)$. Then there exists a positive time $T$ such that the equations (\ref{1.1}) have a unique local solution $(u,b)$ in the space $L^2([0,T];\chi^0(\mr^2))$, which also belongs to $$\tilde{L}^\infty([0,T];\chi^{-1}(\mr^2))\cap L^1([0,T];\chi^1(\mr^2))\cap \mathcal{C}([0,T];\chi^{-1}(\mr^2)).$$
Moreover, there exists a constant $C(\mu,\nu)$ such that if $\|(u_0,b_0)\|_{\chi^{-1}}\leq C(\mu,\nu)$, then the solution is a global one; if $T^*<\infty$ is the maximal time of existence, then
\begin{align}\label{3.1}
\lim_{T\rightarrow T^*}\int_0^{T}\|(u,b)\|_{\chi^{0}}^2\md t=\infty.
\end{align}
\end{lem}
\begin{proof}
It is easy to see that the equations (\ref{1.1}) can be rewritten as the integral form
\be\label{3.2}
\begin{cases}
{ \begin{array}{ll}
u(t,x)=\mathrm{e}^{\mu t\Delta}u_0-\int_0^t\mathrm{e}^{\mu(t-\tau)\Delta}[\mathbb{P}\nabla\cdot
(u\otimes u)-\mathbb{P}\nabla\cdot(b\otimes b)]\md \tau,\\
\\
b(t,x)=\mathrm{e}^{\nu t\Delta}b_0-\int_0^t\mathrm{e}^{\nu(t-\tau)\Delta}[\nabla\cdot
(u\otimes b)-\nabla\cdot(b\otimes u)]\md \tau.
 \end{array} }
\end{cases}
\ee
By applying the Banach contraction mapping principle Lemma \ref{lem2.5}, we will carry out the proof of global or local well-posedness of the Cauchy problem (\ref{3.2}) in the Lei-Lin space $\chi^{-1}(\mr^2)$.
First, one has
\begin{align}\label{3.3}
\nonumber&\|(\mathrm{e}^{\mu t\triangle}u_0,\mathrm{e}^{\nu t\triangle}b_0)\|^2_{L^2(\chi^{0})}\\
\nonumber=&\bigg\|\int_{\mr^2}\mathrm{e}^{-\mu t|\xi|^2}|\hat{u}_0(\xi)|\md \xi\bigg\|^2_{L^2}+\bigg\|\int_{\mr^2}\mathrm{e}^{-\nu t|\xi|^2}|\hat{b}_0(\xi)|\md \xi\bigg\|^2_{L^2}\\
\nonumber\leq& \bigg(\int_{\mr^2}\left(\int_0^t\mathrm{e}^{-2\nu \tau|\xi|^2}|\hat{u}_0(\xi)|^2 \md \tau\right)^{\frac{1}{2}}\md\xi\bigg)^{2}+\bigg(\int_{\mr^2} \left(\int_0^t\mathrm{e}^{-2\nu \tau|\xi|^2}|\hat{b}_0(\xi)|^2 \md \tau\right)^{\frac{1}{2}}\md\xi\bigg)^{2}\\
\leq&\frac1{2\min\{\mu,\nu\}}\|(u_0,b_0)\|^2_{\chi^{-1}}.
\end{align}
Thus, Combining the estimate (\ref{3.3}) with Lemma \ref{lem2.4}, we can obtain that if $\|(u_0,b_0)\|_{\chi^{-1}}\leq\frac{\min\{\mu,\nu\}}{2^\frac32C_0}$, then
\begin{align}\label{3.4}
\|(\mathrm{e}^{\mu t\triangle}u_0,\mathrm{e}^{\nu t\triangle}b_0)\|_{L^2(\chi^{0})}\leq \frac{\min\{\mu,\nu\}^\frac12}{4C_0}<\frac{\min\{\mu,\nu\}^\frac12}{4C}
\end{align}
with $C_0>C$, and we obtain the global solution.

Now, we consider the case of a large initial data $(u_0,b_0)$ in $\chi^{-1}(\mr^2)$. Setting
$$(u_0,b_0)=(u_0^\ell,b_0^\ell)+(u_0^\hbar,b_0^\hbar),$$
where
$$(u_0^\ell,b_0^\ell)\triangleq\mathcal{F}^{-1}\left(\mathbf{1}_{\{|\xi|\leq\rho_{u_0,b_0}\}}(\hat{u}_0,\hat{b}_0)\right) \quad \mbox{and} \quad (u_0^\hbar,b_0^\hbar)\triangleq\mathcal{F}^{-1}\left(\mathbf{1}_{\{|\xi|>\rho_{u_0,b_0}\}}(\hat{u}_0,\hat{b}_0)\right).$$
One fixes some positive real number $\rho_{u_0,b_0}$ such that
\begin{align}\label{3.5}
\|(u_0^\hbar,b_0^\hbar)\|_{\chi^{-1}}=\int_{|\xi|>\rho_{u_0,b_0}}|\xi|^{-1}\left(|\hat{u}_0|+|\hat{b}_0|\right)\md\xi\leq\frac{\min\{\mu,\nu\}}{2^\frac52C_0}.
\end{align}
By (\ref{3.3}) we derived that
\begin{align}\label{3.6}
\|(\mathrm{e}^{\mu t\triangle}u_0,\mathrm{e}^{\nu t\triangle}b_0)\|_{L^2(\chi^{0})}\leq\frac{\min\{\mu,\nu\}^\frac12}{8C_0}+\|(\mathrm{e}^{\mu t\triangle}u_0^\ell,\mathrm{e}^{\nu t\triangle}b_0^\ell)\|_{L^2(\chi^{0})},
\end{align}
and
\begin{align}\label{3.7}
\nonumber&\|(\mathrm{e}^{\mu t\triangle}u_0^\ell,\mathrm{e}^{\nu t\triangle}b_0^\ell)\|_{L^2(\chi^{0})}\\
\nonumber=&\big\|\int_{|\xi|\leq\rho_{u_0,b_0}}|\xi||\xi|^{-1}\left(\mathrm{e}^{-\mu t|\xi|^2}|\hat{u}_0(\xi)|+\mathrm{e}^{-\nu t|\xi|^2}\hat{b}_0(\xi)|\right)\md \xi\big\|_{L^2}\\
\leq&\rho_{u_0,b_0}T^{\frac12}\|(u_0,b_0)\|_{\chi^{-1}}.
\end{align}
Hence, if we choose
\begin{align}\label{3.8}
T\leq\left(\frac{\min\{\mu,\nu\}^\frac12}{8\rho_{u_0,b_0}C_0\|(u_0,b_0)\|_{\chi^{-1}}}\right)^{2},
\end{align}
then we have a unique solution $(u,b)$ in the ball $B(0,\frac{\min\{\mu,\nu\}^\frac12}{2C_0})$ of the space $L^2([0,T];\chi^0(\mr^2))$.

Next, we prove the persistence that if $(u,b)$ is a solution to (\ref{1.1}) in $L^2([0,T];\chi^0(\mr^2))$ with initial data $(u_0,b_0)\in \chi^{-1}(\mr^2)$, then $(u,b)$ also belongs to $$\mathcal{C}([0,T];\chi^{-1}(\mr^2))\cap L^1([0,T];\chi^1(\mr^2)).$$ In fact, let $s=-1$ and $n=2$ in Lemma \ref{2.6} and by Lemma \ref{2.3} the result is followed.

Finally, we prove the blow-up criterion (\ref{3.1}). Assume that we have a solution to the equations (\ref{1.1}) on a time interval $[0,T)$ such that
$$\int_0^{T}\|(u,b)\|_{\chi^{0}}^2\md t<\infty.$$
We claim that the lifespan $T^*$ of $(u,b)$ is larger than $T$. Indeed, due to the estimate (\ref{2.9}) in Lemma \ref{2.6} and  by Lemma \ref{lem2.3}, we have
\begin{align}\label{3.9}
\int_{\mr^n} \sup_{0\leq t\leq T}|\xi|^{-1}(|\hat{u}|+|\hat{b}|)(t,\xi)\md \xi\leq \|(u_0,b_0)\|_{\chi^{-1}}+C\|(u,b)\|_{L^2(\chi^{0})}^2<\infty.
\end{align}
Thus, a positive number $\rho$ exists such that
\begin{align}\label{3.10}
\forall t\in [0,T), \quad \int_{|\xi|>\rho}|\xi|^{-1}(|\hat{u}|+|\hat{b}|)(t,\xi)\md \xi\leq\frac{\min\{\mu,\nu\}}{2^\frac52C_0}.
\end{align}
The condition (\ref{3.8}) now implies that for any $t\in[0,T)$, the lifespan for a solution to (\ref{1.1}) with initial data $\left(u(t),b(t)\right)$ is bounded from below by a positive real number $C$ which is independent of $t$. Thus the lifespan $T^*>T$, and the whole proof of Lemma \ref{lem3.1} is finished.
\end{proof}
\subsection {Proof of Theorem \ref{thm1.1}} Taking the $L^2$ inner products of the equations $(\ref{1.1})_{1,2}$ with $u$ and $b$, respectively, adding the results and integrating by parts, we obtain, for any $t\in (0,+\infty)$, the energy equality
\begin{align}\label{3.11}
\|(u,b)\|_{L^{2}}^{2}+2\mu\int_0^t\|\nabla u\|_{L^{2}}^{2}\md\tau+2\nu\int_0^t\|\nabla b\|_{L^{2}}^{2}\md\tau=\|(u_0,b_0)\|_{L^{2}}^{2},
\end{align}
for details refer to \cite[Theorem 5.14]{B-C-D}. And by Lemma \ref{lem2.2}, it yields that, for any $t\in (0, +\infty)$
\begin{align}\label{3.12}
\nonumber\int_0^t\|(u,b)\|_{\chi^{-\frac12}}^4 \md \tau
\nonumber\leq&\|(u,b)\|_{L^2}^2\int_0^t\|(u,b)\|_{\dot{H}^1}^2 \md \tau\\
\leq&\frac{1}{2\min\{\mu\nu\}}\|(u_0,b_0)\|_{L^2}^4.
\end{align}
According to Lemmas \ref{lem2.6}, \ref{lem2.3}  and \ref{lem2.1}, by the Young's inequality, we deduce that
\begin{align}\label{3.13}
\nonumber&\|(u,b)\|_{\widetilde{L}^\infty(\chi^{-1})}+\mu\int_0^t\|u\|_{\chi^{1}}\md \tau+\nu\int_0^t\|b\|_{\chi^{1}}\md\tau\\
\nonumber\leq&\|(u_0,b_0)\|_{\chi^{-1}}+C\int_0^t\|(u,b)\|_{\chi^{0}}^2\md\tau\\
\nonumber\leq&\|(u_0,b_0)\|_{\chi^{-1}}+C\int_0^t\|(u,b)\|_{\chi^{-\frac12}}^{\frac43}\|(u,b)\|_{\chi^1}^{\frac23}\md \tau\\
\leq&\|(u_0,b_0)\|_{\chi^{-1}}+C\int_0^t\|(u,b)\|_{\chi^{-\frac12}}^4\md\tau+\frac{1}{2}\min\{\mu,\nu\}\int_0^t\|(u,b)\|_{\chi^1}\md\tau.
\end{align}
Inserting the estimate (\ref{3.12}) into (\ref{3.13}) leads to the result
\begin{align}\label{3.14}
\|(u,b)\|_{\widetilde{L}^\infty(\chi^{-1})}+\frac{1}{2}\min\{\mu,\nu\}\int_0^t\|(u,b)\|_{\chi^1}\md\tau
\leq\|(u_0,b_0)\|_{\chi^{-1}}+\frac{C}{2\min\{\mu\nu\}}\|(u_0,b_0)\|_{L^2}^4.
\end{align}
By Lemma \ref{2.1}, we arrive at
\begin{align}\label{3.15}
\int_0^{t}\|(u,b)\|_{\chi^{0}}^2\md \tau\leq\int_0^{t}\|(u,b)\|_{\chi^{-1}}\|(u,b)\|_{\chi^{1}}\md \tau\leq C\left(\mu,\nu,\|(u_0,b_0)\|_{\chi^{-1}},\|(u_0,b_0)\|_{L^2}\right),
\end{align}
for any $0\leq t\leq T$. Thus, the blow-up criterion (\ref{3.1}) implies that the strong solution $(u,b)$ is a global one in $L^2([0,T];\chi^{0}(\mr^2))$ for any $T<\infty$. And by Lemma \ref{2.6} the solution $(u,b)\in\mathcal{C}(\mr^+;\chi^{-1}(\mr^2))\cap L^1(\mr^+;\chi^1(\mr^2))$, which completes the proof of Theorem \ref{thm1.1}.
\subsection {Proof of Theorem \ref{thm1.2}} Subtracting the two equations satisfied by $(u,b)$ and $(v,h)$, respectively, we get
\be\label{4.1}
\begin{cases}
{ \begin{array}{ll}
\partial_{t}w-\mu\Delta w+(w\cdot\nabla) u+(v\cdot\nabla) w+\nabla (p_1-p_2)=(g\cdot\nabla) b+(h\cdot\nabla)g,\\
\partial_{t}g-\nu\Delta g+(w\cdot\nabla) b+(v\cdot\nabla) g=(g\cdot\nabla) u+(h\cdot\nabla) w, \\
 \end{array} }
\end{cases}
\ee
where $(w, g)\triangleq (u-v, b-h)$. Taking the $L^2$ inner products of the equations $(\ref{4.1})_{1,2}$ with $w$ and $g$, respectively, adding the results and by H\"{o}lder's and Young's inequalities and the fact $\|f\|_{L^\infty}\leq\|\hat{f}\|_{L^1}=\|f\|_{\chi^0}$, and
$$\int_{\mr^{2}}(v\cdot\nabla)w\cdot w \md x=0, \ \int_{\mr^{2}}(v\cdot\nabla)g\cdot g\md x=0$$
and $$\int_{\mr^{2}}(h\cdot\nabla)g\cdot w\md x+\int_{\mr^{2}}(h\cdot\nabla)w\cdot g\md x=0,$$
it follows that
\begin{align}\label{4.2}
\nonumber&\frac12\frac{\md}{\md t}\|(w,g)\|_{L^2}^2+\mu\|\nabla w\|_{L^2}^2+\nu\|\nabla g\|_{L^2}^2\\
\nonumber=&-\int_{\mr^{2}}(w\cdot\nabla) u\cdot w\md x+\int_{\mr^{2}}(g\cdot\nabla) b\cdot w\md x-\int_{\mr^{2}}(w\cdot\nabla) b\cdot g\md x+\int_{\mr^{2}}(g\cdot\nabla) u\cdot g\md x\\
\nonumber\leq&C\left(\|\nabla w\|_{L^2}\|w\|_{L^2}\|u\|_{L^\infty}+\|\nabla w\|_{L^2}\|g\|_{L^2}\|b\|_{L^\infty}+\|\nabla g\|_{L^2}\|w\|_{L^2}\|b\|_{L^\infty}+\|\nabla g\|_{L^2}\|g\|_{L^2}\|u\|_{L^\infty}\right)\\
\leq&\frac12\min\{\mu,\nu\}\|(\nabla w,\nabla g)\|_{L^2}^2+C\|(w,g)\|_{L^2}^2\|(u,b)\|_{\chi^0}^2.
\end{align}
Thanks to the Gr\"{o}nwall's inequality, we conclude the proof of (\ref{1.4}). And we thus complete the proof of Theorem \ref{thm1.2}.

\textbf{Acknowledgements} The research of B Yuan
was partially supported by the National Natural Science Foundation
of China (No. 11471103).


\end{document}